\documentclass[12pt,reqno]{amsart}
\usepackage[utf8]{inputenc}
\usepackage{amssymb, amsmath, amsthm, url, mathrsfs}
\usepackage{hyperref}
\usepackage{todonotes}
\usepackage{enumerate}
\usepackage{float}
\usepackage{mathptmx}
\usepackage{caption}
\usepackage{subcaption}

\usepackage{tikz}
\usepackage{tkz-euclide}
\usetikzlibrary{patterns}
\usepackage{algorithm}
\usepackage[noend]{algpseudocode}

\newtheorem{theorem}{Theorem}[section]
\newtheorem*{theorem*}{Theorem}
\newtheorem{lemma}[theorem]{Lemma}
\newtheorem{corollary}[theorem]{Corollary}

\theoremstyle{definition}
\newtheorem{defin}[theorem]{Definition}

\newtheorem{ex}[theorem]{Example}
\theoremstyle{plain}
\newtheorem{proposition}[theorem]{Proposition}

\theoremstyle{remark}
\newtheorem{rem}[theorem]{Remark}

\newcommand{\se}{\operatorname{small}}

\title{Eliahou number, Wilf function and concentration of a numerical semigroup }
\author{Patricio Almir\'on}
\author{Julio Jos\'e Moyano-Fern\'andez}



\thanks{The first author was partially supported by Spanish Goverment, Ministerios de Ciencia e Innovaci\'on y de Universidades MTM2016-76868-C2-1-P. The second author was partially supported by the Spanish Government, Ministerios de Ciencia e Innovaci\'on y de Universidades, grant PGC2018-096446-B-C22, as well as by Universitat Jaume I, grant UJI-B2018-10.}

\address{Instituto de Matemática interdisciplinar (IMI) y departamento de \'{A}lgebra, Geometr\'{i}a y Topolog\'{i}a\\
	Facultad de Ciencias Matem\'{a}ticas\\
	Universidad Complutense de Madrid\\
	28040, Madrid, Spain.}

\email{palmiron@ucm.es}

\address{Universitat Jaume I, Campus de Riu Sec, Departamento de Matem\'aticas \& Institut Universitari de Matem\`atiques i Aplicacions de Castell\'o, 12071
	Caste\-ll\'on de la Plana, Spain}

\email{moyano@uji.es}

\begin{document}

\begin{abstract}
We give an estimate of the minimal positive value of the Wilf function of a numerical semigroup in terms of its concentration. We describe necessary conditions for a numerical semigroup to have negative Eliahou number in terms of its multiplicity, concentration and Wilf function. Also, we show new examples of numerical semigroups with negative Eliahou number. In addition, we introduce the notion of highly dense numerical semigroup; this yields a new family of numerical semigroups satisfying the Wilf conjecture. Moreover, we use the Wilf function of a numerical semigroup to prove that the Eliahou number of a highly dense numerical semigroup is positive under certain additional hypothesis. In particular, these results provide new evidences in favour of the Wilf conjecture.
\end{abstract}

\maketitle

\section{Introduction}

Let $\mathbb{N}$ denote the set of the nonnegative integer numbers. A numerical semigroup $\Gamma$ is an additive sub-semigroup of the monoid $(\mathbb{N},+)$ such that the greatest common divisor of all its elements is equal to $1$. The complement $\mathbb{N}\setminus \Gamma$ is therefore finite, and the elements of that complement are called gaps of $\Gamma$. Moreover, $\Gamma$ is finitely generated and it is not difficult to find a minimal system of generators of $\Gamma$, whose cardinality $e=e(\Gamma)$ is called the embedding dimension of $\Gamma$. The number $c=c(\Gamma)=\max (\mathbb{N}\setminus \Gamma)+1$ is called the conductor of $\Gamma$, and the number of elements in $\Gamma$ which are smaller than $c(\Gamma)$ is said to be the delta-invariant $\delta=\delta (\Gamma)$ of $\Gamma$. The Frobenius number of $\Gamma$ is said to be $f=f(\Gamma):=c(\Gamma)-1$.
\medskip

One of the most challenging open problems in commutative algebra is the Wilf conjecture, proposed by H. Wilf in 1978 \cite{wilf}: it foretells the inequality
$$
c(\Gamma) \leq e(\Gamma) \cdot \delta (\Gamma)
$$
between the conductor and the product of the embedding dimension and the delta invariant of a numerical semigroup $\Gamma$. It is customary to call Wilf number to the difference \(e(\Gamma) \cdot \delta (\Gamma)-c(\Gamma)\) in order to state the Wilf conjecture as the positivity of this quantity.
\medskip

The Wilf conjecture is known to be true in several cases, see e.g.~ Delgado \cite{delgado1}, Dobbs and Matthews \cite{dm}, Eliahou \cite{eliahou}, Eliahou and Fromentin \cite{eliahou-fromentin}, Fromentin and Hivert \cite{fromentin-hivert}, Kaplan \cite{kaplan}, Moscariello and Sammartano \cite{mossam} and Sammartano \cite{samartano}. However this is still an open question, and in the meanwhile several related problems have been treated in order to gain a better understanding of the conjecture (see for example \cite{mdelgado,delgado3,almiyano-wilf}). One of those nearby problems is related to another interesting number that can be associated to a semigroup, the Eliahou number.
\medskip

Eliahou \cite{eliahou} has been able to relate the Wilf conjecture to an invariant \(E(\Gamma)\), nowadays called the Eliahou number associated to a semigroup (cf. \cite{delgado1}; see also subsection \ref{subsec:Eli}) such that \(e(\Gamma) \cdot \delta (\Gamma)-c(\Gamma)\geq E(\Gamma)\). Therefore, any numerical semigroup with positive Eliahou number satisfies the Wilf conjecture. Unfortunately, there exist numerical semigroups with negative Eliahou number \cite{delgado1,eliahou,eliahou-fromentin} and thus the Wilf conjecture is reduced to study those semigroups with negative Eliahou number. 
\medskip

The characterization of those numerical semigroups with negative Eliahou number is a huge challenge, as observed e.g. in \cite{delgado1,delgado3,eliahou-fromentin}, and very few general properties about them are known. One of the main goals of this paper is to establish necessary conditions for a numerical semigroup in order to have negative Eliahou number, see Section \ref{sec:negeli}. In particular, this conditions will allow us to show some new examples of numerical semigroups with negative Eliahou number \ref{ex:negativeEli}. To do so, we will make use of two recent techniques introduced in the study of the Wilf conjecture.
\medskip

On the one hand, the authors proposed in \cite{almiyano-wilf} to associate a function \(W_\Gamma(k):\mathbb{N}\rightarrow \mathbb{Z}\) to the numerical semigroup \(\Gamma\) in order to gain a better understanding of the Wilf conjecture, see e.g. Theorem 3.2 and Theorem 4.1 in \cite{almiyano-wilf} and, in particular, of the conditions leading to the vanishing of the Wilf number. In this direction, we also propose to study the constant
$$
\mu_{\Gamma}:=\min \{k\in \mathbb{N}: W_{\Gamma}(k)\geq 0\}.
$$

In this paper we present a step forward towards the understanding of the conjecture: Eliahou reduced the problem to the study of the Wilf number for semigroups with negative Eliahou number. If we denote by \(e_s\) the number of minimal generators less than \(c(\Gamma),\) our approach will reduce the study of the semigroups with negative Eliahou number to the investigation of semigroups with \(\mu_{\Gamma}\geq e_s\). Concretely, we will show in Theorem \ref{thm:wilfboundeli} that the Eliahou number is bounded below by \(W_\Gamma(e_s).\) Thus, the negativity of the Eliahou number implies \(\mu_{\Gamma}>e_s\). To find examples  with \(\mu_{\Gamma}>e_s\) and positive Eliahou number is not difficult as Example \ref{ex:3} shows; and obviously because  semigroups with negative Eliahou number seem to be rare \cite{delgado1,eliahou-fromentin}. Therefore, the condition \(\mu_{\Gamma}>e_s\) may lead to an easier characterization of the interesting family of semigroups to study.

\medskip

On the other hand, in a recent preprint Rosales et al.~\cite{Rosconce} introduce the concept of concentration of a numerical semigroup: set \(\operatorname{next}_\Gamma(s) :=
\min \{x \in \Gamma \;|\; s < x\}\); the \emph{concentration of a numerical semigroup}
is then defined as 
$$
\mathbf{C}(\Gamma) = \max \{\operatorname{next}_\Gamma(s) - s \;|\; s\in \Gamma\setminus\{0\}\}.
$$

If we call $m(\Gamma):=\min (\Gamma\setminus\{0\})$ the \emph{multiplicity} of the semigroup, it is clear that a numerical semigroup with concentration \(1\) is of the form \(\{0,m(\Gamma),\rightarrow\}\), where the arrow $\rightarrow$ means that from $m(\Gamma)$ on all natural numbers belong to the set. The numerical semigroups with concentration \(2\) have been characterized by Rosales et al.~\cite{Rosconce}.
\medskip

The essence of the Wilf conjecture may be expressed as how often elements of \(\Gamma\) occurr in the integral interval \([0,c]\cap \mathbb{N}\). From this viewpoint, it is natural to ask for semigroups with fixed concentration satisfying the Wilf conjecture. One should obviously expect that smaller concentration should lead to a higher frequency of occurrence of elements of \(\Gamma\) in the integral interval \([0,c]\cap \mathbb{N}\); this is indeed the case as shown in our Theorem \ref{thm:m/keli}.
\medskip

The current manuscript is organized as follows. Section \ref{sect:pre} presents all the required techniques of the theory of numerical semigroups needed for the remainder of the paper ---besides the concentration, namely the Wilf function and the Eliahou number of a semigroup, as well as some partitions of the set of elements in the semigroup due to both Eliahou and Sammartano. Section \ref{sect:pos} appeals to the utility of the concentration in order to show the nonnegativity of the Wilf function, see Propositions \ref{prop:2k} and \ref{prop:k+1}. In Section \ref{sec:negeli} we establish criteria for the negativity resp. positivity of the Eliahou number in terms of the Wilf function resp. concentration of the numerical semigroup: this is the content of Theorem \ref{thm:42} resp.~Theorem \ref{thm:m/keli}. Moreover, we show new examples of semigroups with negative Eliahou number \ref{ex:negativeEli}. In the concluding Section \ref{sec:hdense} we present a family of examples (under the name \emph{highly concentrated} numerical semigroups) that show the utility of the previous results.

\section{Preliminaries}\label{sect:pre}

Let $\Gamma$ be the numerical semigroup generated by $a_1,\ldots , a_e$; this fact will be expressed by writing $\Gamma=\langle a_1,\ldots , a_e \rangle$. We will assume that $a_1,\ldots , a_e$ are a \emph{minimal} system of generators. We will moreover consider that they are ordered with respect to the natural ordering and write \(G:=\{a_1=m<a_2<\cdots<a_e\}\); observe that trivially $m(\Gamma)=a_1$. For generalities on numerical semigroups the reader is referred to the book of Rosales and Garc\'ia-S\'anchez \cite{RosalesGarciaSanchez}; see also the book of Ram\'irez-Alfons\'in \cite{ram}.
\medskip

In the study of the Wilf conjecture there are two features which are relevant for the remainder of the paper, namely the Wilf function and the Eliahou number of a numerical semigroup. We will summarize the fundamentals of these two topics, as well as their implications in the computation of the delta-invariant.

\subsection{The Wilf function}\label{subsec:Wilffun}
In our previous paper \cite{almiyano-wilf}, we defined the map
\[\begin{array}{cccl}
	W_{\Gamma}:&\mathbb{N}&\rightarrow&\mathbb{Z}\\
	&k&\mapsto &W_{\Gamma}(k):=k\delta(\Gamma)-c(\Gamma)
\end{array}\]
as an attempt to give more instruments for the investigation of the Wilf conjecture.  
The function $W_{\Gamma}$ is what we called the \emph{Wilf function} of the semigroup $\Gamma$. As already mentioned, for $k=e(\Gamma)=e$, the nonnegativity $W_{\Gamma}(e)\geq 0$ expresses thus the statement of the Wilf conjecture; indeed, the study of the behaviour of the Wilf function contributes to the understanding of Wilf's conjecture, as shown in \cite{almiyano-wilf}. 
\medskip

In general, $k=e$ is not the minimal value making $W_{\Gamma}(k)$ nonnegative. This means that the Wilf number $W_{\Gamma}(e)$ does not yield in general a sharp bound for the positivity of the Wilf function. From this point of view, it would be certainly interesting to investigate the constant
$$
\mu_{\Gamma}:=\min \{k\in \mathbb{N}: W_{\Gamma}(k)\geq 0\},
$$
where obviously $2\leq \mu_{\Gamma}\leq m$, as mentioned in \cite{almiyano-wilf}. This constant will play a role in the study of the positivity of the Wilf function by the concentration of the semigroup which will be done in Section \ref{sect:pos}.

\subsection{The Eliahou number}\label{subsec:Eli}

An ultimate tool towards the solution of the Wilf conjecture seems to be the Eliahou number, whose definition will be recalled in the sequel. 
\medskip

Let $q:=q(\Gamma)=\lceil\frac{c(\Gamma)}{m(\Gamma)}\rceil$ be the $q$-number of $\Gamma$. We set
\begin{align*}
	\nu(\Gamma)=&\nu= qm-c\\
	\se(\Gamma)=& | \{s\in \Gamma : s<c\} |\\
	G(\Gamma)=G:=&\{a_1,\ldots , a_e\} \\
\end{align*}

Eliahou \cite{eliahou} introduced the following partition of the interval \([-\nu,c+m]\):
\[J_\alpha:=[\alpha m-\nu,(\alpha+1)m-\nu] \ \ \mbox{for} \ \alpha =0, 1, \ldots, q .\]

Of course, the interval \([-\nu,c+m]\) is meant to be $[-\nu,c+m]\cap \mathbb{N}$, but we leave out the intersection with $\mathbb{N}$ in order to discharge the notation, and we will assume this and all occurring intervals to be in $\mathbb{N}$. 
\medskip

The main advantage of Eliahou's partition is that the last subinterval is exactly \(J_q=[c,c+m]\) (we will consider a different partition due to Sammartano \cite{samartano} in Section \ref{sect:pos} in which this property does not hold). 
\medskip

Set \(p_q:=J_q\cap G\) and \(d_q:=[c,c+m]\setminus p_q\). Let us denote by \(e_s:=|G\cap\se(\Gamma)|\) resp. \(e_c:=|p_q|\) the number of minimal generators of the semigroup which are smaller than the conductor resp. bigger than the conductor. Obviously, \(e(\Gamma)=e_s+e_c.\) Eliahou introduced the following invariant \cite{eliahou}, named the \emph{Eliahou number} of $\Gamma$ after him, cf. \cite{delgado1}:

\[E(\Gamma)=e_s\delta(\Gamma)-q|d_q|+\nu.\]

The Eliahou number plays a role in the Wilf conjecture in virtue of the following \cite[Proposition 3.11]{eliahou}:

\begin{theorem}[Eliahou]
	Let \(\Gamma\) be a numerical semigroup, then \(W_\Gamma(e)\geq E(\Gamma)\).
\end{theorem}

It is an important result the fact that negative Eliahou numbers can be effectively attained, see \cite[Corollary~14, Corollary~35]{delgado1}:

\begin{theorem}[Delgado]
	For any \(z\in\mathbb{Z}\) there exist a numerical semigroup with \(E(\Gamma)=z.\) In particular there exist numerical semigroups with arbitrarily negative Eliahou number.
\end{theorem}

\subsection{\(\delta(\Gamma)\) from the Apéry set}
There is a remarkable system of generators ---by no means minimal--- that can be attached to a numerical semigroup $\Gamma$: let $s \in \Gamma\setminus \{0\}$, the Ap\'ery set of $\Gamma$ with respect to $s$ is defined to be the set
\[
\mathrm{Ap}(\Gamma, s)=\{w\in\Gamma\,:\;w-s\notin \Gamma\},
\]
see Ap\'ery \cite{apery}, or also Kunz and Herzog \cite[Lemma 4.2]{KH}.
\medskip

Observe that the cardinality of $\mathrm{Ap}(\Gamma, s)$ is $s$, and that $\mathrm{Ap}(\Gamma, s)=\{w_0<w_1<\dots<w_{s-1}\}$ where $w_i=\min \{ z\in \Gamma : z \equiv i \ \mathrm{mod}\ s\}$; obviously, $w_0=0$. We will always consider the particular case $s=m:=m(\Gamma)$, for which $w_1=a_2$ and $w_{m-1}=c-1+a_1=c+m-1$.
\medskip

In this subsection we will leave Eliahou's partition, and following Sammartano \cite{samartano} we will adopt instead a partition in subintervals of length $m-1$ of the form
\[
I_\alpha:=[\alpha m,(\alpha+1)m-1], \ \mbox{for} \ \alpha =0, 1, \ldots, L,
\]

where $L:=\lfloor \frac{c-1}{m}\rfloor=\lfloor \frac{w_{m-1}}{m}\rfloor-1$ denotes the integer part of the quotient between the conductor of \(\Gamma\) minus $1$ ---the Frobenius number of $\Gamma$--- and its multiplicity. Hence,  we can write \(c-1=Lm+\rho'\) with \(0\leq \rho'\leq m-1\) and \(\rho'\neq 0\) because \(c-1\neq \Gamma\). Therefore, we can rewrite \(c= Lm+\rho\) with \(\rho=\rho'+1\) and \(2\leq\rho\leq m\). Thus we have in particular the following identity.
\begin{lemma}\label{lem:auxL}
	Let \(\Gamma\) be a numerical semigroup with conductor \(c\) and multiplicity \(m\), and set $L:=\lfloor \frac{c-1}{m}\rfloor$. Then, 
	\[L=\Bigg\{\begin{array}{cc}
		\lfloor \frac{c}{m}\rfloor&\text{if \(c\) is not a multiple of \(m\) }\\
		&\\
		\frac{c}{m}-1& \text{if \(c\) is a multiple of \(m\) }
	\end{array}\]
\end{lemma}
In particular, Lemma \ref{lem:auxL} implies that \(q=L+1\) and that the number of subintervals \(J_\alpha\) is one more than the number of subintervals \(I_\alpha.\)
\medskip

Following the notation of \cite{samartano} and \cite{mossam}, for $j=1, \ldots , m-1$ we define
$$
\eta_j=|\{\alpha\in\mathbb{N} : | I_\alpha\cap\Gamma |=j\}|\quad\text{and}\quad n_{\alpha}=|\{s\in \Gamma \cap I_{\alpha}: s<f \}|.
$$

The number $\eta_j$ can be computed from the Ap\'ery set $\mathrm{Ap}(\Gamma, m)$ in the following way:

\begin{lemma}[\cite{samartano}, Proposition 13]\label{lemma:21}
	For any $j=1, \ldots , m-1$ we have
	\[
	\eta_j=\Big \lfloor \frac{w_j}{m}\Big \rfloor-\Big \lfloor \frac{w_{j-1}}{m}\Big \rfloor.
	\]
\end{lemma}

Therefore we can compute \(\delta(\Gamma)\) in terms of the Apéry set as follows:
\begin{proposition}\label{prop:deltaaper}
	Let \(\Gamma\) be a numerical semigroup, then
	\[\delta(\Gamma)=m \Big \lfloor \frac{w_{m-1}}{m} \Big \rfloor - \sum_{j=0}^{m-1} \Big \lfloor \frac{w_j}{m}\Big \rfloor +\rho - m.\]
\end{proposition}
\begin{proof}
	First of all, it is a trivial observation that \(\delta(\Gamma)=n_0+\cdots+n_L\). Thus, an extensive use of the statement in Lemma \ref{lemma:21} shows that
	
	\begin{align*}
		\delta=&\sum_{j=0}^{L} n_j= \sum_{j=1}^{m-1}(\eta_j\cdot j) +\rho-m =  \sum_{j=1}^{m-1} \bigg (\sum_{i=j}^{m-1} \eta_i \bigg ) +\rho-m \\
		=& \sum_{j=1}^{m-1} \bigg ( \sum_{i=j}^{m-1} \Big \lfloor \frac{w_i}{m}\Big \rfloor-\Big \lfloor \frac{w_{i-1}}{m}\Big \rfloor \bigg) + \rho-m =  \sum_{j=1}^{m-1} \Big ( \Big \lfloor \frac{w_{m-1}}{m}\Big \rfloor-\Big \lfloor \frac{w_{j-1}}{m}\Big \rfloor \Big ) +\rho - m \\
		= & (m-1) \Big \lfloor \frac{w_{m-1}}{m}\Big \rfloor -\sum_{j=1}^{m-2} \Big \lfloor \frac{w_j}{m}\Big \rfloor+\rho - m\\
		= & m \Big \lfloor \frac{w_{m-1}}{m} \Big \rfloor - \sum_{j=0}^{m-1} \Big \lfloor \frac{w_j}{m}\Big \rfloor +\rho - m,
	\end{align*}
	as desired.
\end{proof}

\section{Positivity of the Wilf function associated to the concentration}\label{sect:pos}

As already mentioned in subsection \ref{subsec:Wilffun}, the study of the invariant \(\mu_\Gamma\) seems to be difficult: in fact, this should help to give a sharper inequality than that in the Wilf conjecture. In this direction, very few is known. This section is devoted to use the notion of concentration to show upper bounds for the invariant \(\mu_\Gamma\). 
\medskip

To do so, we are first going to give estimates for the \(\delta\)-invariant and the embedding dimension in terms of the multiplicity of the semigroup and its concentration.


\begin{proposition}\label{prop:deltaestim}
	Let \(\Gamma\) be a numerical semigroup with concentration $\mathbf{C}(\Gamma) =k$ and \(c=Lm+\rho\), then 
	\[\delta(\Gamma)\geq\frac{(L-1)m+\rho}{k}+1.\]
\end{proposition}

\begin{proof}
	Let us denote by \(A_\alpha:=I_\alpha\cap \Gamma.\) For \(1\leq \alpha\leq L-1\) consider the set
	\[A'_\alpha:=\{b_1:=\alpha m<\cdots<b_s:=(\alpha +1)m\;|\;b_i\in \Gamma\}.\]
	
	Thus, \(|A_\alpha|=|A'_\alpha|-1=s-1.\) On the other hand, since we are assuming concentration \(k\), we have that
	\[m=(b_s-b_{s-1})+\cdots+(b_2-b_{1})\leq k(s-1).\]
	Hence,
	\(|A_\alpha|=s-1\geq \frac{m}{k}.\)
	
	A simple observation shows that \(|A_0|=1\), and so
	\[\delta(\Gamma)=1+\sum_{\alpha=1}^{L-1}|A_\alpha|+(|B_L|-1),\]
	where \(B_L:=\{x_1:=Lm<\cdots<x_t:=c=Lm+\rho\}.\)
	
	Again, since the concentration is assumed to be \(k\), we have \(|B_L|-1\geq \rho/k.\) Therefore
	
		\[\delta(\Gamma)=1+\sum_{\alpha=1}^{L-1}|A_\alpha|+|B_L|\geq \frac{(L-1)m+\rho}{k}+1,\]
		as we wished.
\end{proof}
\begin{rem}
	Observe that the bound in Proposition \ref{prop:deltaestim} is sharp: consider $$\Gamma =W_{m,q}=\langle m, qm+1, \ldots , qm+(m-1)\rangle$$ for integers $m,q$ such that $m>1$ and $q>0$. These semigroups have concentration \(\mathbf{C}(\Gamma)=k=m\), and moreover \(\rho=m\) and \(L=q-1.\) Also it is easy to check that \(\delta=L+1\). Thus,
	
	\[L+1=\delta\geq \frac{(L-1)m+m}{m}+1=L+1.\]
	The semigroups $W_{m,q}$ are indeed very interesting in the context of Wilf conjecture: the authors proved in \cite[Theorem 4.8]{almiyano-wilf} the following characterization of the nonpositivity of the Wilf function: \(\Gamma=W_{m,q}\) for \(q\geq 1\) if and only if \(W_{\Gamma}(k)\leq 0\) for all \(1\leq k\leq m\).
\end{rem}
\begin{rem}
	In the particular case of \(k=2\) and \(c>2m,\) Rosales et al. \cite[Lemma 2]{Rosconce} show the inequality \(\delta\geq m/2+2\). The assumption \(c>2m\) leads to \(L\geq 2\), hence our Proposition \ref{prop:deltaestim} covers their result. 
\end{rem}

\begin{proposition}\label{prop:embedding}
	Let \(\Gamma\) be a numerical semigroup with concentration \(\mathbf{C}(\Gamma)=k\) and conductor \(c>2m\), then \(e_s\geq m/k\). In particular, the embedding dimension is bounded below by $m/k$, i.e. \(e\geq m/k.\)
\end{proposition}
\begin{proof}
	Any element of the interval \(I_1\cap\Gamma\) is a minimal generator of the semigroup. Hence \(e\geq |I_1\cap\Gamma|\geq m/k\), where the last inequality holds because of the arguments in the proof of Proposition \ref{prop:deltaestim}.
\end{proof}
The nonnegativity of the Wilf function can be related to the concentration of the semigroup in the following manner.

\begin{proposition}\label{prop:2k}
	Let \(\Gamma\) be a numerical semigroup with concentration \(\mathbf{C}(\Gamma)=k\), then \(W_\Gamma(2k)\geq 0\). In particular, \(2k\geq\mu_\Gamma.\)
\end{proposition}

\begin{proof}
	Let us write \(c=Lm+\rho\) with $L:=\lfloor \frac{c-1}{m}\rfloor$ and \(2\leq \rho\leq m.\)
	By Proposition \ref{prop:deltaestim} we have \(k\delta(\Gamma)\geq(L-1)m+\rho\). On the other hand, since \(G\subset\Gamma\setminus\{0\}\) and \(e=|G|\), Proposition \ref{prop:embedding} implies \(k\delta(\Gamma)\geq m.\) Therefore
	\[W_\Gamma(2k)=2k\delta(\Gamma)-c\geq (L-1)m+\rho+m-c=0,\]
	and we are done.
\end{proof}

The inequalities in Proposition \ref{prop:2k} can be improved by adding additional hypothesis:

\begin{proposition}\label{prop:k+1}
	Let \(\Gamma\) be a numerical semigroup with concentration \(\mathbf{C}(\Gamma)=k\). If \(\delta(\Gamma)\geq m-k\), then \(W_\Gamma(k+1)\geq 0\). In particular, \(k+1\geq\mu_\Gamma.\)
\end{proposition}
\begin{proof}
	Let us write \(c=Lm+\rho\) with $L:=\lfloor \frac{c-1}{m}\rfloor$ and \(2\leq \rho\leq m.\)
	By Proposition \ref{prop:deltaestim} we have \(k\delta(\Gamma)\geq(L-1)m+\rho+k\). Moreover, since \(\delta(\Gamma)\geq m-k\) by hypothesis, the claim follows.
\end{proof}

\section{On the negativity of the Eliahou number}\label{sec:negeli}
The negativity of Eliahou number poses an interesting question within the theory of numerical semigroups: the semigroups having negative Eliahou number seem to be rare and infrequent, as already observed in several works \cite{delgado1,delgado3,eliahou-fromentin}. In this section, we first present a lower bound for the Eliahou number in terms of the Wilf function. This contrasts with the fact that the Eliahou number attains any integer value and allows us to provide a necessary condition for its negativity in terms of the Wilf function. In addition, we continue the section investigating Eliahou numbers in semigroups with fixed concentration. This will allow us to provide a necessary condition for its negativity in terms of the concentration.

\subsection{Eliahou number vs Wilf function}
As already mentioned in Subsection \ref{subsec:Eli}, Delgado showed that the Eliahou number can attain any integer value. The main problem for the computation of the Eliahou number is that the Eliahou partition \(J_\alpha\) defining Eliahou numbers does not coincide with the one defined by Samartano \(I_\alpha\), and Samartano's partition allows an easier calculation of \(\delta\), as Proposition \ref{prop:deltaaper} witnesses. 
\medskip

Our main idea in this subsection is to give a range of the possible values of the Eliahou number by considering the Wilf function; this will allow us to check only the properties of Wilf function in order to study semigroups with a prescribed Eliahou number.

\begin{theorem}\label{thm:wilfboundeli}
	Let \(\Gamma\) be a numerical semigroup with embedding dimension \(e\). Preserving the notation of Subsection \ref{subsec:Eli}, we have the inequalities
	
	\[W_\Gamma(e)\geq E(\Gamma)\geq W_\Gamma(e_s).\]
\end{theorem}
\begin{proof}
	The first inequality is due to Eliahou \cite[Proposition 3.11]{eliahou}. The second inequality is deduced from the fact that \(|d_q|\leq m,\) so that
	\[W_\Gamma(e)\geq E(\Gamma)=e_s\delta(\Gamma)-q|d_q|+\nu\geq e_s\delta(\Gamma)-qm+\nu=W_\Gamma(e_s),\]
	which is our assertion.
\end{proof}
\begin{ex}\label{ex:3}
	Before presenting some computations, we establish the following standard notation: write \(S=\langle x_1\dots,x_s\rangle_r\) for the minimal semigroup that contains \(\{x_1\dots,x_s\}\) and all the integers greater than or equal to \(r\). This notation is widely used e.g. by Delgado in \cite{delgado1}.
	
	\medskip
	 According to the computations done with the functions in GAP \cite{gap}, the numerical semigroup \(\Gamma:=\langle 30,42,51\rangle_{290}\) has \(W_\Gamma(e_s)<0,\) \(\mu_\Gamma=5,\) \(c=290,\) \(e=23,\) \(\delta=65\) and 
	
	\[W_\Gamma(\mu_{\Gamma})=35<E(\Gamma)=105<W_\Gamma(e)=1205.\]
\end{ex}
Theorem \ref{thm:wilfboundeli} yields a necessary condition for the negativity of Eliahou number in terms of the Wilf function.
\begin{theorem}\label{thm:42}
	Let \(\Gamma\) be a numerical semigroup with Eliahou number \(E(\Gamma)<0\). Then, 
	
	\[W_\Gamma(e)<e_c\delta.\]
	
	In particular, \(\mu_\Gamma>e_s.\)
\end{theorem}
\begin{proof}
	We begin by observing that 
	\[W_\Gamma(1)=-\sum_{j=0}^{m-1} \Big \lfloor \frac{w_j}{m}\Big\rfloor;\]
	this follows by Proposition \ref{prop:deltaaper} and by the fact that \(c=Lm+\rho\) with \(L=\lfloor \frac{w_{m-1}}{m}\rfloor-1\).
	
	 On the other hand, from the linearity of Wilf function we have	\(W_\Gamma(e_s)=(e_s-1)\delta+W_\Gamma(1).\) Moreover, since \(E(\Gamma)<0\), Theorem \ref{thm:wilfboundeli} implies \(W_\Gamma(e_s)<0.\) All this together yields
	\begin{align*}
		W_\Gamma(e-1)&=W_\Gamma(e_s+e_c-1)=(e_s-1)\delta+W_\Gamma(e_c)<-W_\Gamma(1)+W_\Gamma(e_c)=(e_c-1)\delta,
	\end{align*}
	which establishes the desired inequality.
\end{proof}

\subsection{Positivity of the Eliahou number associated to the concentration}
Once we have shown a necessary condition for the negativity of the Eliahou number obtained thanks to the Wilf function, our purpose now is to give a necessary condition for the negativity of Eliahou number in terms of the concentration of the semigroup. To do so, we first need to prove the following.
\begin{theorem}\label{thm:m/keli}
	Let \(\Gamma\) be a numerical semigroup with multiplicity \(m\) and concentration \(\mathbf{C}(\Gamma)=k.\) Write \(c=Lm+\rho\) with \(2\leq \rho\leq m,\) and assume that \(c>2m\). If \(m/k^2>(L+1)/(L-1),\) then \(E(\Gamma)\geq 0\).
\end{theorem}
\begin{proof}
	First of all, observe that Lemma \ref{lem:auxL} implies that \(L+1=\lceil c/m\rceil=q\). Also \(|d_q|\leq m.\) Therefore
	\[E(\Gamma)\geq e_s\delta(\Gamma)-(L+1)m.\]
	On the other hand, Proposition \ref{prop:deltaestim} together with Proposition \ref{prop:embedding} give us
	\[e_s\delta(\Gamma)-(L+1)m\geq\bigg(\frac{m}{k}\bigg)\bigg(\frac{(L-1)m+\rho+k}{k}\bigg)-(L+1)m.\]
	
	Since \(\rho,k\geq 0\), the claim follows from the hypothesis \(m/k^2>(L+1)/(L-1).\)
\end{proof}

Theorem \ref{thm:m/keli} gives us an easy-to-handle condition which implies the positivity of the Eliahou number. In contrast to the examples of negative Eliahou number given by Delgado \cite{delgado1}, Eliahou \cite{eliahou}, and Fromentin \cite{eliahou-fromentin}, our condition only assumes the knowledge of the multiplicity, the concentration and the conductor of the semigroup. In this way we do not need to compute neither the embedding dimension nor the \(\delta\)-invariant in our case. This leads to the following necessary condition for a semigroup to be a semigroup with negative Eliahou number.

\begin{corollary}\label{cor:en}
	Let \(\Gamma\) be a numerical semigroup with multiplicity \(m\) and concentration \(\mathbf{C}(\Gamma)=k\), and write \(c=Lm+\rho\) with \(2\leq \rho\leq m.\) If \(E(\Gamma)< 0\), then \(m/k^2<(L+1)/(L-1)\).
\end{corollary}

\begin{rem}
	It is not difficult to check that all the semigroups defined by Delgado in \cite{delgado1} that have negative Eliahou number satisfy the inequality \(m/k^2<(L+1)/(L-1)\) and \(k<m.\)
\end{rem}

Here it is natural to ask whether the condition \(m/k^2<(L+1)/(L-1)\) is too restrictive. This seems not to be the case: it is quite easy to construct numerical semigroups satisfying the mentioned inequality. The general trick to find them is to observe that \(1<(L+1)/(L-1)<2\) if \(L\geq 4\) and \((L+1)/(L-1)\geq 2\) if \(1\leq L\leq 3.\) Now, we have two options: either we choose a big multiplicity in order to allow bigger concentrations, or we choose directly small concentrations. Let us illustrate this behaviour with some examples computed with the aid of GAP \cite{gap,gapdelgado}: 
\begin{ex}\label{ex:2}
	Let \(A:=\{1000+25\cdot k\;|0\leq k\leq 39\}\). Let \(\Gamma\) be the numerical semigroup minimally generated by \(A \cup \{1507,1899, 13765, 13790, 13815\}.\) The multiplicity of $\Gamma$ is \(m(\Gamma)=1000,\) the conductor is \(c=13741=13\cdot1000+741\), and the concentration \(\mathbf{C}(\Gamma)=25.\) Thus \(L=13\) and the conditions of Theorem \ref{thm:m/keli} are fulfilled, therefore $\Gamma$ has positive Eliahou number. 
\end{ex}

\begin{ex}\label{ex:1}
	Let us consider the numerical semigroup defined by \[\Gamma=\langle50,55,60,65,70,73,77,81,86,91,96,194,199\rangle.\]
 We see that \(c=190\) and it has concentration \(\mathbf{C}(\Gamma)=5\). Then it fulfils the hypothesis of Theorem \ref{thm:m/keli} and so \(E(\Gamma)>0.\) Moreover, since \(\delta=66>50\) it satisfies the conditions of Proposition \ref{prop:k+1} so \(W(6)\geq 0\). An easy computation shows that \(E(\Gamma)=544\) and \(W(6)=206.\) It is also easily seen that \(\mu_\Gamma=3.\)
 	
	On the other hand, the type of $\Gamma$ is $17$, and it is neither symmetric nor pseudo-symmetric, according to the computations done with the routines in GAP \cite{gap}. Moreover, it is easily checked that it does not fulfils any of the conditions of the main theorems of \cite{mossam,samartano}.
\end{ex}

\subsection{Examples of semigroups with negative Eliahou number}
A few examples of numerical semigroups with negative Eliahou number are known. Some of them already appeared in Eliahou's paper \cite{eliahou}. Those are the unique numerical semigroups with negative Eliahou number and \(c-\delta\leq 60.\) Later, Delgado \cite[Sections~3~and~4]{delgado1} provided several families of numerical semigroups with negative Eliahou number and \(e_s=3.\) In fact, these families offer examples with arbitrarily large negative Eliahou number. Moreover, Delgado showed a few examples with \(e_s=4,5\) in \cite[Tables 6 and 7]{delgado1}. More recently, Eliahou and Fromentin \cite{eliahou-fromentin} presented new families of numerical semigroups with negative Eliahou number, all of them with \(c=4m.\)
\medskip

It is not difficult to check that all the examples provided by Delgado, Eliahou and Fromentin satisfy the conditions of Theorem \ref{thm:42} and Corollary \ref{cor:en}. We wonder whether these necessary conditions may help to find new examples of numerical semigroups with negative Eliahou number. This is the case; in fact we present now a few of them: it is straightforward to check that they do not belong to the above collections of Delgado~resp.~Eliahou and Fromentin \cite{eliahou,eliahou-fromentin}, since in our examples we have \(e_s=4\) and \(c\geq 5m\); to the best of the authors' knowledge, these are not mentioned in the literature.

\begin{ex}\label{ex:negativeEli}
In the following table we show eight numerical semigroups with negative Eliahou number, \(e_s=4\) and concentrations \(70,100.\)
	\begin{table}[H]
		\begin{center}
			\begin{tabular}{c|cccccc} 
				$\Gamma$ & $E(\Gamma)$ & $\mathbf{C}(\Gamma)$  & $e_i$ & $\mu_i$ & $W_i(e_i)$ & $W_i(\mu_i)$   \\
				\hline\hline
				$\langle 100, 170, 171, 176 \rangle_{599}$ & $-1$ & $70$ & $71$ & \textbf{13} & $2880 $ & $38$ \\
				$\langle 100, 270, 272, 275  \rangle_{998}$ & $-2$ & $100$ & $70$ & \textbf{15} &  $4882$ & $52$\\
				$\langle 100, 270, 271, 175  \rangle_{999}$ & $-3$ & $100$ & $70$ & \textbf{12} &  $4881$ & $9$\\
				$\langle 100, 270, 273, 275  \rangle_{1000}$ & $-4$ & $100$ & $70$ & \textbf{12} &  $4880$ & $8$\\
				$\langle 100, 170, 173, 174 \rangle_{597}$ & $-5$ & $70$ & $70$ & \textbf{13} & $2833 $ & $40$ \\
				$\langle 100, 170, 172, 175 \rangle_{598}$ & $-6$ & $70$ & $70$ & \textbf{13} & $2832$ & $39$ \\
				$\langle 100, 170, 173, 175 \rangle_{599}$ & $-7$ & $70$ & $70$ & \textbf{13} & $2831$ & $38$ \\
				$\langle 100, 170, 172, 175  \rangle_{600}$ & $-8$ & $70$ & $70$ & \textbf{13} &  $2830$ & $37$\\
			
			\end{tabular}
			\vspace{0.3cm}
			\caption{Some semigroups with negative Eliahou number.}  \label{tab:table1}
		\end{center}
	\end{table}
Different combinations of the minimal generators and conductors of the examples of Table \ref{tab:table1} allowed us to find \(36\) numerical semigroups with Eliahou number within the interval \([-8,-1].\) Those semigroups are of two types: 
\begin{enumerate}
	\item [\textbf{Type 1}] \(\langle 100,170,a,b\rangle_c\) with \(a,b\in [171,176],\) \(c\in [597,600].\) These semigroups have \(e_s\in\{3,4\}\) and \(c>5m.\)
	\item [\textbf{Type 2}] \(\langle 100,270,a,b\rangle_c\) with \(a,b\in [271,276],\) \(c\in [997,1000].\) These semigroups have \(e_s\in\{3,4\}\) and \(c>9m.\)
\end{enumerate}

There are \(18\) numerical semigroups of type \(1\) and negative Eliahou number and \(18\) numerical semigroups of type \(2\) and negative Eliahou number. All of them can be computed with the help of GAP \cite{gap,gapdelgado} by using the following codes:

\begin{algorithm}[H]
	\caption{Code to compute the \(18\) numerical semigroups negative Eliahou number of type \(1\) }\label{gapcode1}
	\begin{algorithmic}[1]
		\State $L1:=[593..602];;$
		\State $L2:=[171..180];;$
		
	\For { $ i \ \text{in} \ [1..Length(L1)]$ } \do\\
	 \For { $j \ \text{in}\ [1..Length(L2)]$ } \do\\
	\For { $k \ \text{in}\ [1..Length(L2)]$ } \do\\
	 \State $G:=NumericalSemigroup(Union([100,170,L2[j],L2[k]],[L1[i]..900]));$
	 
	\If { $EliahouNumber(G)<0$ }
	
	\State $Print([100,170,L2[j],L2[k]],"\setminus n");$
	
	\State $Print(Conductor(G),"\setminus n");$
	
	\State $Print(EliahouNumber(G),"\setminus n");$
	
	\State $Print("=============================","\setminus n");$
	
	\State fi;
	
	\EndIf
	\State od;
	\EndFor
	\State od;
	
	\EndFor
	\State od;
	
	\EndFor

	\end{algorithmic}
\end{algorithm}
\begin{algorithm}[H]
	\caption{Code to compute the \(18\) numerical semigroups negative Eliahou number of type \(2\) }\label{gapcode2}
	\begin{algorithmic}[1]
		\State $L1:=[993..1005];;$
		\State $L2:=[271..280];;$
		
		\For { $ i \ \text{in} \ [1..Length(L1)]$ } \do\\
		\For { $j \ \text{in}\ [1..Length(L2)]$ } \do\\
		\For { $k \ \text{in}\ [1..Length(L2)]$ } \do\\
		\State $G:=NumericalSemigroup(Union([100,270,L2[j],L2[k]],[L1[i]..9000]));$
		
		\If { $EliahouNumber(G)<0$ }
		
		\State $Print([100,270,L2[j],L2[k]],"\setminus n");$
		
		\State $Print(Conductor(G),"\setminus n");$
		
		\State $Print(EliahouNumber(G),"\setminus n");$
		
		\State $Print("=============================","\setminus n");$
		
		\State fi;
		
		\EndIf
		\State od;
		\EndFor
		\State od;
		
		\EndFor
		\State od;
		
		\EndFor
		
	\end{algorithmic}
\end{algorithm}
\end{ex}

\section{Highly dense numerical semigroups}\label{sec:hdense}

To finish, let us present a class of numerical semigroups satisfying the Wilf conjecture. To do so, we will make use of the results of Section \ref{sect:pos}. Furthermore, we will prove that under certain restrictions they also have positive Eliahou number. We need first the notion of \emph{highly dense} numerical semigroup:

\begin{defin}
	We say that \(\Gamma\) is highly dense if one of the following two conditions is satisfied:
	\begin{enumerate}
		\item \(\Gamma\) has concentration less or equal than \(2\).
		\item \(\Gamma\) has concentration less or equal than \(e(\Gamma)/2\) and \(4\leq e(\Gamma)\).
	\end{enumerate}
\end{defin}
Examples \ref{ex:2} and \ref{ex:1} show already highly dense numerical semigroups with concentration $\mathbf{C}(\Gamma)=5$ resp. $\mathbf{C}(\Gamma)=25$. We employ the terminology \emph{highly dense} due to the fact that small concentrations lead to higher number of elements of the numerical semigroup in the interval \([0,c]\), as Proposition \ref{prop:deltaestim} shows.  

	

\medskip

By definition and the discussion of Section \ref{sect:pos} we have the following.
\begin{proposition}
	Let \(\Gamma\) be a highly dense numerical semigroup. Then \(W_\Gamma(e)\geq 0\).
\end{proposition}
\begin{proof}
	This is an straightforward consequence of Proposition \ref{prop:2k}.
\end{proof}
Therefore, highly dense numerical semigroups provide a new family of numerical semigroups satisfying Wilf's conjecture. Moreover, we can use Theorem \ref{thm:wilfboundeli} to show that ---under certain additional hypothesis--- highly dense numerical semigroups have positive Eliahou number.

\begin{corollary}
	Let \(\Gamma\) be a numerical semigroup with \(\mathbf{C}(\Gamma)=k\geq 2\), with conductor \(c>2m\) and satisfying \(e_s\geq 2k\). Then \(E(\Gamma)\geq 0.\)
	
	In particular, any highly dense numerical semigroup with \(e_s\geq 2k\) has positive Eliahou number.
\end{corollary}
\begin{proof}
	By Theorem \ref{thm:wilfboundeli} we have \(E(\Gamma)\geq W_\Gamma(e_s).\) Since \(e_s\geq 2k\), the linearity of the Wilf function together with Proposition \ref{prop:2k} shows 
	\[E(\Gamma)\geq W_\Gamma(e_s)\geq W_\Gamma(2k)\geq 0,\]
	as desired.
\end{proof}
\begin{corollary}
	Let \(\Gamma\) be a numerical semigroup with concentration \(k\geq 2,\) with \(e_s\geq k+1\), delta-invariant \(\delta(\Gamma)\geq m-k\), and conductor \(c>2m\). Then \(E(\Gamma)\geq 0.\)
	
	In particular, any highly dense numerical semigroup with \(e_s\geq k+1\)  and \(\delta(\Gamma)\geq m-k\) has positive Eliahou number.
\end{corollary}
\begin{proof}
	From Theorem \ref{thm:wilfboundeli} we have \(E(\Gamma)\geq W_\Gamma(e_s).\) Since \(e_s\geq k+1\) and \(\delta(\Gamma)\geq m-k\), the linearity of the Wilf function together with Proposition \ref{prop:k+1} shows 
	\[E(\Gamma)\geq W_\Gamma(e_s)\geq W_\Gamma(k+1)\geq 0,\]
	as wished.
\end{proof}

\end{document}